\documentclass[11pt]{amsart}

\usepackage{amsfonts}
\usepackage{tikz}
\usepackage{enumerate}
\usepackage{amsmath,amscd}
\usepackage{amsmath,amsthm,amssymb,anysize}
\usepackage[numbers,sort&compress]{natbib}
\usepackage{color}
\usepackage{graphicx}
\usepackage[all,cmtip]{xy}
\usepackage[colorlinks,citecolor=blue]{hyperref}

\usepackage[normalem]{ulem}

\theoremstyle{definition}
\newtheorem{thm}{Theorem}

\newtheorem{rem}[thm]{Remark}
\newtheorem{prop}[thm]{Proposition}
\newtheorem{defn}[thm]{Definition}

\newcommand{\bea}{\begin{eqnarray}}
\newcommand{\eea}{\end{eqnarray}}
\newcommand{\nn}{\nonumber}

\newcommand{\cred}[1]{\textcolor{black}{#1}}

\marginsize{3cm}{3cm}{2.5cm}{2.2cm}%左右上下
\numberwithin{equation}{section}
\numberwithin{thm}{section}
\def\Z{\mathbb{Z}}

\def\C{\mathbb{C}}

\def\R{\mathbb{R}}

\def\CP{\mathbb{CP}}

\def\cM{\mathcal{M}}

\title{Studying complex manifolds by using groups $G_{n}^{k}$ and $\Gamma_{n}^{k}$}



\author{Vassily Olegovich Manturov}

\address{Bauman Moscow State Technical University, Moscow 105005, Russia; Novosibirsk State University, Novosibirsk 630090, Russia}
\email{vomanturov@yandex.ru}

\author{Zheyan Wan}

\address{Yau Mathematical Sciences Center, Tsinghua University, Beijing 100084, China}
\email{wanzheyan@mail.tsinghua.edu.cn}

\date{}

\begin{document}

%\date{}
\maketitle

\vskip 10pt

\begin{abstract}
In the present paper, we study several complex manifolds by using the following idea.
First, we construct a certain moduli space and study the fundamental group
of this space. This fundamental group is naturally mapped to the groups $G_{n}^{k}$ and $\Gamma_{n}^{k}$.
This is the step towards ``complexification'' of the $G_{n}^{k}$ and $\Gamma_{n}^{k}$ approach first developed in \cite{2019arXiv190508049M}.
\end{abstract}

\section{Introduction}\label{sec:introduction}

In 2015, the first named author \cite{2015arXiv150105208O} defined a 2-parameter family of groups $G_n^k$ for natural numbers $n>k$. Those groups may be regarded as a certain generalization of braid groups. Study of the connection between the groups $G_n^k$ and dynamical systems led to the discovery of the following fundamental principle: 

\emph{``If dynamical systems describing the motion of $n$ particles possess a nice codimension one property governed by exactly $k$ particles, then these dynamical systems admit a topological invariant valued in $G_n^k$''. }

Later the first named author and his colleagues \cite{2019arXiv190211238M,2019arXiv191202695F} introduced and studied the second family of groups, denoted by $\Gamma_n^k$ ($n\ge k\ge4$), which are closely related to triangulations of manifolds. 
Here the vertices of the triangulation play the role of particles and generators correspond to flips of the
triangulation changing its combinatorics.

The nice codimension one property for developing the $\Gamma_n^k$ theory is:

\emph{``$k$ points of the configuration lie on a sphere of dimension $k-3$ and there are no points inside the sphere''.}
 
The main idea behind the $G_{n}^{k}$ and  $\Gamma_{n}^{k}$ approach is that in order
to study fundamental groups of a certain moduli space, we create some natural ``walls''
(codimension one sets) which correspond to generators, with codimension two sets
(intersections of walls) corresponding to the relations. This approach was successfully applied
in \cite{2019arXiv190508049M} for studying manifolds, dynamics, and invariants.

The main difference between $G_{n}^{k}$ and $\Gamma_{n}^{k}$ approach is the construction of
walls: in the case of $\Gamma_{n}^{k}$ the walls are chosen according to some {\em local}
condition. For example, when we study points in $\R^{2}$, the map to $G_{n}^{4}$ deals with
quadruples of points belonging to the same circle (or line), whence $\Gamma_{n}^{4}$ deals
with quadruples of {\em neighbouring} points belonging to the same circle.
 
\section{Basic definitions}\label{sec:definitions}

First, we recall the definition of the groups $G_n^k$ given in \cite{2015arXiv150105208O}.

Consider the following $\binom{n}{k}$ generators $a_m$ where $m$ runs the set of all unordered $k$-tuples $m_1,\dots, m_k$, whereas each $m_i$ are pairwise distinct numbers from $\{1,\dots,n\}$.

For each unordered $(k+1)$-tuple $U$ of distinct indices $u_1,\dots,u_{k+1}\in\{1,\dots,n\}$, consider the $k+1$ sets $m^j=U\setminus\{u_j\}$, $j=1,\dots,k+1$. With $U$, we associate the relation
\bea\label{tetra}
a_{m^1}\cdots a_{m^{k+1}}=a_{m^{k+1}}\cdots a_{m^1}.
\eea
For two tuples $U$ and $\bar{U}$, which differ by order reversal, we get the same relation.

Thus, we totally have $\frac{(k+1)!\binom{n}{k+1}}{2}$ relations.
We shall call them the \emph{tetrahedron relations}.

For $k$-tuples $m,m'$ with $\text{Card}(m\cap m')<k-1$, consider the \emph{far commutativity relation}:
\bea\label{far}
a_ma_{m'}=a_{m'}a_m.
\eea
Note that the far commutativity relation can occur only if $n>k+1$.

Besides that, for all multiindices $m$, we write down the following relation
\bea\label{third}
a_m^2=1.
\eea
\begin{defn}
The $k$-free braid group $G_n^k$ is defined as the quotient group of the free group generated by all $a_m$ for all multiindices $m$ by relations \eqref{tetra}, \eqref{far} and \eqref{third}.
\end{defn}

The $n$-strand planar braid group $B_n(\R^2)$ is generated by $\sigma_1,\dots,\sigma_{n-1}$ and is defined by the relations
$$\sigma_i\sigma_{i+1}\sigma_i=\sigma_{i+1}\sigma_i\sigma_{i+1},\ i=1,2,\dots,n-2,$$
$$\sigma_i\sigma_j=\sigma_j\sigma_i,\ |i-j|>1.$$

There is a homomorphism from $B_n(\R^2)$ to the symmetric group $S_n$, which sends $\sigma_i$ to the transposition $(i,i+1)$. Its kernel is the $n$-strand planar pure braid group $PB_n(\R^2)$. This group is generated by $b_{ij}$, $1\le i<j\le n$, where
$$b_{i,i+1}=\sigma_i^2,$$
$$b_{ij}=\sigma_{j-1}\sigma_{j-2}\cdots\sigma_{i+1}\sigma_i^2\sigma_{i+1}^{-1}\cdots\sigma_{j-2}^{-1}\sigma_{j-1}^{-1},\ i+1<j\le n.$$

The following two propositions (\ref{G_n^3} and \ref{G_n^4}) are based on the two {\em nice codimension one properties:
three points are collinear
and
four points belong to the same circle/line}.

The map from the pure braid group $PB_{n}$ to $G_{n}^{3}, G_{n}^{4}$ is obtained by studying a generic braid and writing down the generators $a_{ijk}$ (respectively, $a_{ijkl}$) where indices correspond to the numbers of points which correspond to the codimension 1 condition.

For each different indices $i,j$, $1\le i,j\le n$, we consider the element $c_{i,j}$ in the group $G_n^3$ to be the product 
$$c_{i,j}=\prod_{k=j+1}^na_{i,j,k}\prod_{k=1}^{j-1}a_{i,j,k}.$$

\begin{prop}[Proposition 3 of \cite{2019arXiv190508049M}]\label{G_n^3}
The map 
$$b_{i,j}\mapsto c_{i,i+1}^{-1}\cdots c_{i,j-1}^{-1}c_{i,j}^2c_{i,j-1}\cdots c_{i,i+1},\ i<j,$$
defines a homomorphism $\phi_n: PB_n(\R^2)\to G_n^3$.
\end{prop}

Let $a_{\{i,j,k,l\}}$, $1\le i,j,k,l\le n$, be the generators of the group $G_n^4$, $n>4$.

Let $1\le i<j\le n$. Consider the elements
\bea
c_{ij}^I=\prod_{p=2}^{j-1}\prod_{q=1}^{p-1}a_{\{i,j,p,q\}},
\eea
\bea
c_{ij}^{II}=\prod_{p=1}^{j-1}\prod_{q=1}^{n-j}a_{\{i,j-p,j,j+q\}},
\eea
\bea
c_{ij}^{III}=\prod_{p=1}^{n-j+1}\prod_{q=0}^{n-p+1}a_{\{i,j,n-p,n-q\}},
\eea
\bea
c_{ij}=c_{ij}^{II}c_{ij}^Ic_{ij}^{III}.
\eea

\begin{prop}[Proposition 4 of \cite{2019arXiv190508049M}]\label{G_n^4}
The map 
$$b_{i,j}\mapsto c_{i,i+1}\cdots c_{i,j-1}c_{i,j}^2c_{i,j-1}^{-1}\cdots c_{i,i+1}^{-1},\ i<j,$$
defines a homomorphism $\psi_n: PB_n(\R^2)\to G_n^4$.
\end{prop}

Next, we recall the definition of the groups $\Gamma_n^k$ given in \cite{2019arXiv190211238M,2019arXiv191202695F}.

\begin{defn}
Let $4\le k\le n$. The group $\Gamma_n^k$ is the group with generators 
$$a_{P,Q}, \ P,Q\subset\{1,\dots,n\}, \ P\cap Q=\emptyset, \ |P\cup Q|=k, \ |P|,|Q|\ge2$$
and the relations:
\begin{enumerate}
\item 
$a_{Q,P}=a_{P,Q}^{-1}$;
\item
\emph{far commutativity}: $a_{P,Q}a_{P',Q'}=a_{P',Q'}a_{P,Q}$ for each generators $a_{P,Q}, a_{P',Q'}$ such that 
$$|P\cap(P'\cup Q')|<|P|,\ |Q\cap(P'\cup Q')|<|Q|,$$
$$|P'\cap|P\cup Q)|<|P'|,\ |Q'\cap|P\cup Q)|<|Q'|;$$
\item
\emph{$(k+1)$-gon relations}: for any standard Gale diagram (see \cite{2019arXiv191202695F} for the definition) $\bar{Y}$ of order $k+1$ and any subset $M=\{m_1,\dots,m_{k+1}\}\subset\{1,\dots,n\}$,
$$\prod_{i=1}^{k+1}a_{M_R(\bar{Y},i), M_L(\bar{Y},i)}=1,$$ 
where $M_R(\bar{Y},i)=\{m_j\}_{j\in R_{\bar{Y}}(i)}$, $M_L(\bar{Y},i)=\{m_j\}_{j\in L_{\bar{Y}}(i)}$.
\end{enumerate}
\end{defn}

In particular,
\begin{defn}
The group $\Gamma_n^4$ is the group generated by
$$\{d_{(ijkl)}\mid \{i,j,k,l\}\subset\{1,\dots,n\},|\{i,j,k,l\}|=4\}$$ subject to the following relations:
\begin{enumerate}
\item
$d_{(ijkl)}^2=1$ for $\{i,j,k,l\}\subset\{1,\dots,n\}$,
\item
$d_{(ijkl)}d_{(stuv)}=d_{(stuv)}d_{(ijkl)}$ for $|\{i,j,k,l\}\cap\{s,t,u,v\}|<3$.
\item
$d_{(ijkl)}d_{(ijlm)}d_{(jklm)}d_{(ijkm)}d_{(iklm)}=1$ for distinct $i,j,k,l,m$.
\item
$d_{(ijkl)}=d_{(kjil)}=d_{(ilkj)}=d_{(klij)}=d_{(jkli)}=d_{(jilk)}=d_{(lkji)}=d_{(lijk)}$ for distinct $i,j,k,l$.
\end{enumerate}
\end{defn}

Let us denote
$$d_{\{p,q,(r,s)_s\}}=\left\{\begin{array}{ll}d_{(pqrs)}&\text{if }p<q<s,\\
d_{(prsq)}&\text{if }p<s<q,\\d_{(rspq)}&\text{if }s<p<q,\\d_{(qprs)}&\text{if }q<p<s,\\
d_{(qrsp)}&\text{if }q<s<p,\\d_{(rsqp)}&\text{if }s<q<p.
\end{array}\right.$$

Let us denote $k\in\{p,q,r\}$ such that 
$\text{min}\{p,q,r\}<k<\text{max}\{p,q,r\}$ by $\text{mid}\{p,q,r\}$.

Let us define $\gamma_{\{p,q,(i,j)_j\}}$ as follows:
\begin{enumerate}
\item
If $\text{min}\{p,q,j\}<i<\text{mid}\{p,q,j\}$ or $i>\text{max}\{p,q,j\}$, then
\bea
&&\gamma_{\{p,q,(i,j)_j\}}\nn\\
&=&\left\{\begin{array}{ll}d_{\{p,q,(i,j)_j\}},&\text{if }\text{min}\{p,q,j\}-\text{mid}\{p,q,j\}+\text{max}\{p,q,j\}-2=0,\\
1,&\text{if }\text{min}\{p,q,j\}-\text{mid}\{p,q,j\}+\text{max}\{p,q,j\}-2\ne0.\end{array}\right.
\eea
\item
If $i<\text{min}\{p,q,j\}$ or $\text{mid}\{p,q,j\}<i<\text{max}\{p,q,j\}$, then
\bea
&&\gamma_{\{p,q,(i,j)_j\}}\nn\\
&=&\left\{\begin{array}{ll}d_{\{p,q,(i,j)_j\}},&\text{if }\text{min}\{p,q,j\}-\text{mid}\{p,q,j\}+\text{max}\{p,q,j\}-2=1,\\
1,&\text{if }\text{min}\{p,q,j\}-\text{mid}\{p,q,j\}+\text{max}\{p,q,j\}-2\ne1.\end{array}\right.
\eea
\end{enumerate}

Let $b_{ij}\in PB_n(\R^2)$, $1\le i<j\le n$, be a generator.
Consider the elements
$$\Delta_{i,(i,j)}^I=\prod_{p=2}^{j-1}\prod_{q=1}^{p-1}\gamma_{\{p,q,(i,j)_j\}},$$
$$\Delta_{i,(i,j)}^{II}=\prod_{p=1}^{j-1}\prod_{q=1}^{n-j}\gamma_{\{(j-p),(j+q),(i,j)_j\}},$$
$$\Delta_{i,(i,j)}^{III}=\prod_{p=1}^{n-j-1}\prod_{q=0}^{p-1}\gamma_{\{(n-p),(n-q),(i,j)_j\}},$$
$$\Delta_{i,(i,j)}=\Delta_{i,(i,j)}^{II}\Delta_{i,(i,j)}^I\Delta_{i,(i,j)}^{III}.$$

Now we define $\xi_n:PB_n(\R^2)\to\Gamma_n^4$ by
$$\xi_n(b_{ij})=\Delta_{i,(i,(i+1))}\cdots\Delta_{i,(i,(j-1))}\Delta_{i,(i,j)}\Delta_{i,(j,i)}\Delta_{i,((j-1),i)}^{-1}\cdots\Delta_{i,((i+1),i)}^{-1},$$
for $1\le i<j\le n$.

\begin{thm}[Theorem 20 of \cite{2019arXiv190508049M}]\label{Gamma_n^4}
The map $\xi_n:PB_n(\R^2)\to\Gamma_n^4$, which is defined above, is a homomorphism.
\end{thm}

The essence of Theorem \ref{Gamma_n^4} is as follows. We  consider points on $\R^{2}$ as vertices of the triangulation, where three points form a triangle if and only if their circumscribed circle contains no other point.

The generators of the group $\Gamma_{n}^{4}$ correspond to those situations where the combinatorial structure of the triangulation changes: four points belong to the same circle whose interior contains no other point.

Here $d_{(ijkl)}$ corresponds to the quadrilateral where $i$ and $k$ are opposite and $j$ and $l$ are opposite.

\begin{rem}\label{spherical}
It is well-known that the center of $PB_n(\R^2)$ is $Z(PB_n(\R^2))=\Z$ generated by $(\sigma_{1}\cdots \sigma_{n-1})^{n}$.
The maps $PB_{n}\to G_n^3,G_n^4, \Gamma_{n}^{4}$ constructed in \cite{2019arXiv190508049M} take the ``rotation'' element $(\sigma_{1}\cdots \sigma_{n-1})^{n}$ to 1,
thus being maps from the quotient group $PB_n(\R^2)/Z(PB_n(\R^2))$. 
Let $PB_n(S^2)$ denote the $n$-strand spherical pure braid group. 
Since there is a short exact sequence \cite{farbmarg}
\bea\label{exact}
1\to \Z/2\Z\to PB_n(S^2)\xrightarrow{\rho_n} PB_{n-1}(\R^2)/Z(PB_{n-1}(\R^2))\to 1,
\eea
we get homomorphisms from the $n$-strand spherical pure braid group $PB_n(S^2)$ to $G_{n-1}^3$, $G_{n-1}^4$, and $\Gamma_{n-1}^4$ via composition of $\phi_{n-1}$, $\psi_{n-1}$, and $\xi_{n-1}$ with $\rho_n$. We denote them by $\phi_n'$, $\psi_n'$, and $\xi_n'$ respectively.

\end{rem}

\section{Outline of the paper}
In this paper, we will first consider $n$ distinct copies of $\CP^1$ in $\CP^{2}$ with degree 1 (defined by linear equations) and study their ``restricted'' moduli space (section \ref{sec:CP^1}). In the restricted moduli space, any two distinct copies of $\CP^1$ have exactly one intersection point, and no three distinct copies of $\CP^1$ have exactly one common intersection point.
So when we consider one copy of $\CP^1$, there are $n-1$ intersection points (with other copies of $\CP^1$) moving on the considered $\CP^1$, and no two points coincide. In addition, if we give the $n$ copies of $\CP^1$ an order, then we get a continuous map from the ``restricted'' moduli space of $n$ ordered distinct copies of $\CP^1$ in $\CP^2$ to the moduli space of $n-1$ ordered distinct points on a $\CP^1=S^2$. This continuous map induces a group homomorphism from the fundamental group of the ``restricted'' moduli space to the $(n-1)$-strand spherical pure braid group. 
By the results in section \ref{sec:definitions}, we get group homomorphisms from the fundamental group of the ``restricted'' moduli space to the groups $G_{n-2}^3$, $G_{n-2}^4$ and $\Gamma_{n-2}^4$.

Next, we will consider $n$ distinct copies of $\CP^m$ in $\CP^{m+1}$ with degree 1 (defined by linear equations) and study their ``restricted'' moduli space (section \ref{sec:CP^m}).
Similarly, we can inductively get group homomorphisms from the fundamental group of the ``restricted'' moduli space of $n$ distinct copies of $\CP^m$ in $\CP^{m+1}$ to the fundamental group of the ``restricted'' moduli space of $n-1$ distinct copies of $\CP^{m-1}$ in $\CP^{m}$. Finally, we combine these results and the results in section \ref{sec:CP^1} and get group homomorphisms from the fundamental group of the ``restricted'' moduli space to the groups $G_{n-m-1}^3$, $G_{n-m-1}^4$ and $\Gamma_{n-m-1}^4$.

As we mentioned in section \ref{sec:introduction}, these group homomorphisms to the groups $G_n^k$ and $\Gamma_n^k$ are invariants which will be useful to our future work.

\section{Main theorem for $\CP^{1}$ in $\CP^{2}$ with degree 1}\label{sec:CP^1}

In this section, we assume that all $\CP^1$ in $\CP^2$ have degree 1.

Consider $n$ distinct copies of $\CP^1$ in $\CP^{2}$. Since $$\CP^{2}=\{[x_0:x_1:x_{2}]\mid (x_0,x_1,x_{2})\in\C^{3}\setminus\{0\}\},$$ each $\CP^1$ can be defined by the equation 
$a_{0,i}x_0+a_{1,i}x_1+a_{2,i}x_{2}=0$ where $(a_{0,i},a_{1,i},a_{2,i})\in\C^{3}\setminus\{0\}$. So each $\CP^1$ corresponds to a point $[a_{0,i}:a_{1,i}:a_{2,i}]$ in $\CP^{2}$ and any other $\CP^1$ correspond to other points in $\CP^{2}$.

For any two distinct $\CP^1$, because $[a_{0,i}:a_{1,i}:a_{2,i}]\ne[a_{0,j}:a_{1,j}:a_{2,j}]$,
 the rank of the matrix $\left(\begin{array}{ccc}a_{0,i}&a_{1,i}&a_{2,i}\\a_{0,j}&a_{1,j}&a_{2,j}\end{array}\right)$ is 2.
So the intersection of each two distinct $\CP^1$ is a point since the linear space of the solutions of the linear equations
\bea
\left\{\begin{array}{l}a_{0,i}x_0+a_{1,i}x_1+a_{2,i}x_{2}=0\\a_{0,j}x_0+a_{1,j}x_1+a_{2,j}x_{2}=0\end{array}\right.
\eea
is of complex dimension $1$.

\begin{defn}
The moduli space $\cM_n^{1}$ is the restricted configuration space of $n$ ordered distinct $\CP^{1}$ in $\CP^{2}$ such that there is no three $\CP^1$ which have exactly one common intersection point.
\end{defn}

\begin{thm}\label{M_n^1}
There is a sequence of homomorphisms:
\begin{center}
\begin{tikzpicture}
\node at (0,1) {$\pi_1(\cM_{n}^1)$};
\node at (4,1) {$PB_{n-1}(S^2)$};

\node at (7,2) {$G_{n-2}^3$,};
\cred{\node at (7,1) {$G_{n-2}^4$,};}
\node at (7,0) {$\Gamma_{n-2}^4$,};

\draw[->] (5,1) -- (6.5,2);
\cred{\draw[->] (5,1) -- (6.5,1);}
\draw[->] (5,1) -- (6.5,0);

\draw[->] (1,1) -- (3,1);

\node[above] at (2,1) {$d_{n,i}^1$};

\node[above] at (5.75,1.75) {$\phi_{n-1}'$};
\cred{\node[above] at (5.75,0.75) {$\psi_{n-1}'$};}
\node[below] at (5.75,0.5) {$\xi_{n-1}'$};

\end{tikzpicture}
\end{center}
Here $1\le i\le n$, $\phi_{n-1}'=\phi_{n-2}\circ \rho_{n-1}$, $\psi_{n-1}'=\psi_{n-2}\circ \rho_{n-1}$, and $\xi_{n-1}'=\xi_{n-2}\circ \rho_{n-1}$ where $\rho_{n-1}:PB_{n-1}(S^2)\to PB_{n-2}(\R^2)/Z(PB_{n-2}(\R^2))$ is defined in \eqref{exact}. 
\end{thm}

\begin{proof}
Consider $n$ ordered distinct $\CP^1$ in $\CP^2$ such that there is no three $\CP^1$ which have exactly one common intersection point. 
Denote them by $\CP^1_i$ for $1\le i\le n$.
For each $\CP^1_i$, there are $n-1$ intersection points with other $\CP^1_j$ ($i\ne j$).
Since there is no three $\CP^1$ which have exactly one common intersection point,
the $n-1$ intersection points are pairwise distinct. So we get $n-1$ distinct points in $\CP^1=S^2$. 

Hence we get continuous maps $f_{n,i}^1$ from the moduli space $\cM_n^{1}$ to the configuration space of $n-1$ ordered distinct points in $S^2$. 
These continuous maps induce 
homomorphisms $d_{n,i}^1$ from $\pi_1(\cM_n^{1})$ to the $n-1$ strand spherical pure braid group.

There is a homomorphism $\rho_{n-1}:PB_{n-1}(S^2)\to PB_{n-2}(\R^2)/Z(PB_{n-2}(\R^2))$ (see \eqref{exact}). 
There are homomorphisms $\phi_{n-2}$  from the $n-2$ strand planar pure braid group to $G_{n-2}^3$, $\psi_{n-2}$ from the $n-2$ strand planar pure braid group to $G_{n-2}^4$ and 
$\xi_{n-2}$ from the $n-2$ strand planar pure braid group to
$\Gamma_{n-2}^4$ (Proposition \ref{G_n^3}, Proposition \ref{G_n^4}, and Theorem \ref{Gamma_n^4}). All these homomorphisms map the center of $PB_{n-2}(\R^2)$ to 1, so they induce homomorphisms from $PB_{n-2}(\R^2)/Z(PB_{n-2}(\R^2))$ to $G_{n-2}^3$, $G_{n-2}^4$, and $\Gamma_{n-2}^4$ (Remark \ref{spherical}), we still denote them by $\phi_{n-2}$, $\psi_{n-2}$, and $\xi_{n-2}$.
 
We denote the composition of $\phi_{n-2}$, $\psi_{n-2}$, and $\xi_{n-2}$ with $\rho_{n-1}$ by $\phi_{n-1}'$, $\psi_{n-1}'$, and $\xi_{n-1}'$ respectively.
They are homomorphisms from the $n-1$ strand spherical pure braid group to $G_{n-2}^3$, $G_{n-2}^4$, and $\Gamma_{n-2}^4$.

 Thus we have proved the theorem.

\end{proof}

\section{Main theorem for $\CP^{m}$ in $\CP^{m+1}$ with degree 1}\label{sec:CP^m}

In this section, we assume that all $\CP^m$ in $\CP^{m+1}$ have degree 1, namely they are defined by linear equations.

Consider $n$ distinct copies of $\CP^m$ in $\CP^{m+1}$. Since $$\CP^{m+1}=\{[x_0:x_1:\cdots:x_{m+1}]\mid (x_0,x_1,\dots,x_{m+1})\in\C^{m+2}\setminus\{0\}\},$$ each $\CP^m$ can be defined by the equation 
$a_{0,i}x_0+a_{1,i}x_1+\cdots+a_{m+1,i}x_{m+1}=0$ where $(a_{0,i},a_{1,i},\dots,a_{m+1,i})\in\C^{m+2}\setminus\{0\}$. So each $\CP^m$ corresponds to a point $[a_{0,i}:a_{1,i}:\cdots:a_{m+1,i}]$ in $\CP^{m+1}$ and any other $\CP^m$ correspond to other points in $\CP^{m+1}$.

For any two distinct $\CP^m$ in $\CP^{m+1}$, because $[a_{0,i}:a_{1,i}:\cdots:a_{m+1,i}]\ne[a_{0,j}:a_{1,j}:\cdots:a_{m+1,j}]$, the rank of the matrix $\left(\begin{array}{cccc}a_{0,i}&a_{1,i}&\cdots&a_{m+1,i}\\a_{0,j}&a_{1,j}&\cdots&a_{m+1,j}\end{array}\right)$ is 2. So
the intersection of each two distinct $\CP^m$ is a $\CP^{m-1}$ since the linear space of the solutions of the linear equations
\bea
\left\{\begin{array}{l}a_{0,i}x_0+a_{1,i}x_1+\cdots+a_{m+1,i}x_{m+1}=0\\a_{0,j}x_0+a_{1,j}x_1+\cdots+a_{m+1,j}x_{m+1}=0\end{array}\right.
\eea
is of complex dimension $m$.

\begin{defn}
The moduli space $\cM_n^{m}$ is the restricted configuration space of $n$ ordered distinct $\CP^{m}$ in $\CP^{m+1}$ such that any $m+2$ copies of $\CP^m$ are in general position.
Namely,
there are no three $\CP^m$ which have exactly one common intersection $\CP^{m-1}$, there are no four $\CP^m$ which have exactly one common intersection $\CP^{m-2}$, $\dots$, and there are no $m+2$ copies of $\CP^m$ which have exactly one common intersection point.
\end{defn}
\begin{rem}
By the projective duality, $n$ distinct copies of $\CP^m$ in $\CP^{m+1}$ are in one-to-one correspondence with $n$ distinct points in $\CP^{m+1}$. Any $m+2$ copies of $\CP^m$ are in general position if and only if the corresponding $m+2$ points are in general position.
\end{rem}

\begin{thm}\label{M_n^m}
There is a sequence of homomorphisms:

\begin{center}
\begin{tikzpicture}
\node at (0,0) {$\pi_1(\cM_{n}^m)$};

\node at (0,-2) {$\pi_1(\cM_{n-1}^{m-1})$};

\node at (0,-4) {$\vdots$};

\node at (0,-6) {$\pi_1(\cM_{n-m+1}^1)$};

\node at (0,-8) {$PB_{n-m}(S^2)$};

\node at (-2,-10) {$G_{n-m-1}^3$,};
\cred{\node at (0,-10) {$G_{n-m-1}^4$,};}
\node at (2,-10) {$\Gamma_{n-m-1}^4$,};

\draw[->] (0,-0.5) -- (0,-1.5);
\draw[->] (0,-2.5) -- (0,-3.5);
\draw[->] (0,-4.5) -- (0,-5.5);

\draw[->] (0,-6.5) -- (0,-7.5);

\draw[->] (0,-8.5) -- (-2,-9.5);
\cred{\draw[->] (0,-8.5) -- (0,-9.5);}
\draw[->] (0,-8.5) -- (2,-9.5);

\node[right] at (0,-1) {$d_{n,i(n)}^m$};
\node[right] at (0,-3) {$d_{n-1,i(n-1)}^{m-1}$};
\node[right] at (0,-5) {$d_{n-m+2,i(n-m+2)}^2$};

\node[right] at (0,-7) {$d_{n-m+1,i(n-m+1)}^1$};

\node[left] at (-1.3,-9) {$\phi'_{n-m}$};
\cred{\node[left] at (0.7,-9) {$\psi'_{n-m}$};}
\node[right] at (1.3,-9) {$\xi'_{n-m}$};

\end{tikzpicture}
\end{center}

Here $1\le i(k)\le k$ for $n-m+1\le k\le n$,
$\phi_{n-m}'=\phi_{n-m-1}\circ \rho_{n-m}$, $\psi_{n-m}'=\psi_{n-m-1}\circ \rho_{n-m}$, and $\xi_{n-m}'=\xi_{n-m-1}\circ \rho_{n-m}$ where $\rho_{n-m}:PB_{n-m}(S^2)\to PB_{n-m-1}(\R^2)/Z(PB_{n-m-1}(\R^2))$ is defined in \eqref{exact}. 

\end{thm}

\begin{proof}
Consider $n$ ordered distinct $\CP^m$ in $\CP^{m+1}$ in $\cM_n^m$.
Denote them by $\CP^m_i$ for $1\le i\le n$.
 For each $\CP^m_i$, there are $n-1$ intersections $\CP^{m-1}$ with other $\CP^m_j$ ($i\ne j$).
Since there are no three $\CP^m$ whose intersection is the same $\CP^{m-1}$,
the $n-1$ intersection $\CP^{m-1}$ are pairwise distinct. So we get $n-1$ distinct copies of $\CP^{m-1}$ in $\CP^m$.

Since there are no four $\CP^m$ whose intersection is the same $\CP^{m-2}$, $\dots$, and there are no $m+2$ copies of $\CP^m$ which have exactly one common intersection point, this implies that there are no three $\CP^{m-1}$ whose intersection is the same $\CP^{m-2}$, $\dots$, and there are no $m+1$ copies of $\CP^{m-1}$ which have exactly one common intersection point. So the $n-1$ ordered distinct $\CP^{m-1}$ in $\CP^m$ are actually in $\cM_{n-1}^{m-1}$.

Hence we get continuous maps $f_{n,i}^m$ from the moduli space $\cM_n^{m}$ to the moduli space $\cM_{n-1}^{m-1}$. 
These continuous maps induce homomorphisms $d_{n,i}^m$ from $\pi_1(\cM_n^{m})$ to $\pi_1(\cM_{n-1}^{m-1})$.

Thus by Theorem \ref{M_n^1}, we have proved the theorem.

\end{proof}

\section*{Acknowledgements}
The first named author is supported by the Laboratory of Topology and Dynamics, Novosibirsk State University (grant No. 14.Y26.31.0025 of the government of the Russian Federation). The second named author is supported by the Shuimu Tsinghua Scholar Program.

 \bibliographystyle{abbrv}
\bibliography{G_n^k-Gamma_n^k.bib}

\end{document}